\documentclass[11pt]{article}
\usepackage{amsfonts}
\usepackage{latexsym}
\usepackage{amsmath}
\usepackage{amssymb}
\usepackage{amsthm}
\usepackage{amsmath}
\usepackage{verbatim}
\usepackage{bibspacing}

\newtheorem{thm}{Theorem}[section]
\newtheorem*{thm*}{Theorem}
\newtheorem{cor}[thm]{Corollary}

\newtheorem{lem}[thm]{Lemma}

\newtheorem*{prop*}{Proposition}
\newtheorem{question}[thm]{Question}

\newtheorem*{conj*}{Conjecture}

\newtheorem*{dfn*}{Definition}
\theoremstyle{definition}
\newtheorem{rem}[thm]{\textbf{Remark}}
\newtheorem*{rmk*}{Remark}
\newtheorem*{fact*}{Fact}

\theoremstyle{proof}

\numberwithin{equation}{section}

\newcommand{\vol}{\textrm{Vol}}
\newcommand{\norm}[1]{\left\Vert#1\right\Vert}

\newcommand{\abs}[1]{\left\vert#1\right\vert}
\newcommand{\set}[1]{\left\{#1\right\}}
\newcommand{\brac}[1]{\left(#1\right)}
\newcommand{\scalar}[1]{\left \langle #1 \right \rangle}

\newcommand{\Real}{\mathbb{R}}

\newcommand{\eps}{\epsilon}

\newcommand{\C}{\mathcal{C}}

\newcommand{\II}{\text{II}}

\def\XXint#1#2#3{{\setbox0=\hbox{$#1{#2#3}{\int}$}
\vcenter{\hbox{$#2#3$}}\kern-.5\wd0}}

\newlength{\defbaselineskip}
\newcommand{\setlinespacing}[1]           {\setlength{\baselineskip}{#1 \defbaselineskip}}

\begin{document}

\date{}

\title{Sharp Poincar\'e-type inequality for the Gaussian measure on the boundary of convex sets}
\author{Alexander V. Kolesnikov\textsuperscript{1} and Emanuel Milman\textsuperscript{2}}

\footnotetext[1]{Faculty of Mathematics, Higher School of Economics, Moscow, Russia. 
Supported by RFBR project 14-01-00237 and the DFG project CRC 701.
This study (research grant No 14-01-0056) was supported by The National Research University -- Higher School of Economics' Academic Fund Program in 2014/2015.
Emails: akolesnikov@hse.ru, sascha77@mail.ru.}

\footnotetext[2]{Department of Mathematics, Technion - Israel
Institute of Technology, Haifa 32000, Israel. 
The research leading to these results is part of a project that has received funding from the European Research Council (ERC) under the European Union's Horizon 2020 research and innovation programme (grant agreement No 637851). In addition, E.M. was supported by BSF (grant no. 2010288) and Marie-Curie Actions (grant no. PCIG10-GA-2011-304066). Email: emilman@tx.technion.ac.il.}

\maketitle

\begin{abstract}
A sharp Poincar\'e-type inequality is derived for the restriction of the Gaussian measure on the boundary of a convex set. In particular, it implies a Gaussian mean-curvature inequality and a Gaussian iso-second-variation inequality. The new inequality is nothing but an infinitesimal equivalent form of Ehrhard's inequality for the Gaussian measure. While Ehrhard's inequality does not extend to general $CD(1,\infty)$ measures, we formulate a sufficient condition for the validity of Ehrhard-type inequalities for general measures on $\Real^n$ via a certain property of an associated Neumann-to-Dirichlet operator.  
\end{abstract}

\section{Introduction}

We consider Euclidean space $(\Real^n,\scalar{\cdot,\cdot})$ equipped with the standard Gaussian measure $\gamma = \Psi_\gamma dx$, $\Psi_\gamma(x) = (2 \pi)^{-n/2} \exp(-\abs{x}^2/2)$. 
Let $K \subset \Real^n$ denote a convex domain with $C^2$ smooth boundary and outer unit-normal field $\nu = \nu_{\partial K}$. The second fundamental form $\II = \II_{\partial K}$ of $\partial K$ at $x \in \partial K$ is as usual (up to sign) defined by $\II_x(X,Y) = \scalar{\nabla_X \nu, Y}$, $X,Y \in T_x \partial K$. The quantities:
\[
H(x) := tr(\II_x) ~,~ H_\gamma (x) := H(x) - \scalar{x ,\nu(x)} ~,
\]
are called the mean-curvature and \emph{Gaussian} mean-curvature of $\partial K$ at $x \in \partial K$, respectively. It is well-known that $H$ governs the first variation of the (Lebesgue) boundary-measure $\vol_{\partial K}$ under the normal-map $t \mapsto \exp(t \nu)$, and similarly $H_\gamma$ governs the first variation of the Gaussian boundary-measure $\gamma_{\partial K} :=\Psi_\gamma \vol_{\partial K}$, see e.g. \cite{KolesnikovEMilmanReillyPart2} or Section \ref{sec:2}.

Recall that the Gaussian isoperimetric inequality of Borell \cite{Borell-GaussianIsoperimetry} and Sudakov--Tsirelson \cite{SudakovTsirelson} asserts that if $E$ is a half-plane with $\gamma(E) = \gamma(K)$, then $\gamma_{\partial K}(\partial K) \geq \gamma_{\partial E}(\partial E)$ (in fact, this applies not just to convex sets but to all Borel sets, with an appropriate interpretation of Gaussian boundary measure). In other words:
\[
\gamma_{\partial K}(\partial K) \geq I_\gamma(\gamma(K)) 
\]
with equality for half-planes, 
where $I_\gamma : [0,1] \rightarrow \Real_+$ denotes the Gaussian isoperimetric profile, given by $I_\gamma := \varphi \circ \Phi^{-1}$ with $\varphi(t) = \frac{1}{\sqrt{2 \pi}} \exp(-t^2/2)$ and $\Phi(t) = \int_{-\infty}^t \varphi(s) ds$. Note that $I_\gamma$ is concave and symmetric around $1/2$, hence it is increasing on $[0,1/2]$ and decreasing on $[1/2,1]$. 

\medskip

Our main result is the following new Poincar\'e-type inequality for the Gaussian boundary-measure on $\partial K$:
\begin{thm} \label{thm:main}
For all convex $K$ and $f \in C^1(\partial K)$ for which the following expressions make sense, we have:
\begin{equation} \label{eq:main}
\int_{\partial K} H_\gamma f^2 d\gamma_{\partial K} - (\log I_\gamma)'(\gamma(K)) \brac{\int_{\partial K} f d\gamma_{\partial K}}^2  \leq \int_{\partial K} \scalar{\II_{\partial K}^{-1} \nabla_{\partial K} f , \nabla_{\partial K} f} d\gamma_{\partial K} .
\end{equation}
Here $\nabla_{\partial K} f$ denotes the gradient of $f$ on $\partial K$ with its induced metric, and
$(\log I_\gamma)'(v) = - \Phi^{-1}(v) / I_\gamma(v)$. 
\end{thm}

This inequality is already interesting for the constant function $f \equiv 1$:
\begin{cor}[Gaussian mean-curvature inequality] \label{cor:intro-mean-curvature}
\begin{equation} \label{eq:constant-f}
\int_{\partial K} H_\gamma d\gamma_{\partial K} \leq (\log I_\gamma)'(\gamma(K)) \gamma_{\partial K}(\partial K)^2 .
\end{equation}
In particular, if $\gamma(K) \geq 1/2$ then necessarily $\int_{\partial K} H_\gamma d\gamma_{\partial K} \leq 0$.
\end{cor} 
The latter inequality is sharp, yielding an equality when $K$ is any half-plane $E$. Indeed, 
since  $I_\gamma(\gamma(E)) = \gamma_{\partial E}(\partial E)$, it is enough to note that $E = (-\infty,t] \times \Real^{n-1}$ has constant Gaussian mean-curvature $H_\gamma = -t = (\log \varphi)'(t) = I_\gamma'(\gamma(E))$.

\medskip

More surprisingly, we will see in Section \ref{sec:4} that Corollary \ref{cor:intro-mean-curvature} in fact implies the Gaussian isoperimetric inequality (albeit only for convex sets). Furthermore, we have:

\begin{cor}[Gaussian iso-curvature inequality] \label{cor:intro-iso-curvature}
If $E$ is a half-plane with $\gamma(E) = \gamma(K) \geq 1/2$, then the following iso-curvature inequality holds:
\[
\int_{\partial K} H_\gamma d\gamma_{\partial K}  \leq \int_{\partial E} H_\gamma d\gamma_{\partial E} \; (\; \leq 0 ) .  
\]
\end{cor}
\begin{proof}
This is immediate from (\ref{eq:constant-f}), the Gaussian isoperimetric inequality $\gamma_{\partial K}(\partial K) \geq \gamma_{\partial E}(\partial E)$, the assumption that $(\log I_\gamma)'(\gamma(K)) \leq 0$, and the equality in (\ref{eq:constant-f}) for half-planes. 
\end{proof}
Clearly, by passing to complements, the latter corollary yields a reverse inequality when applied to $K$, the complement to a convex set $C$ satisfying $\gamma(K) \leq 1/2$ (since $\partial K = \partial C$ with reverse orientation and thus their generalized mean-curvature simply changes sign). It is also easy to check that a reverse inequality holds when $K$ is a small Euclidean (convex) ball centered at the origin. 
It is probably unreasonable to expect that a reverse inequality holds for all convex $K$ with $\gamma(K) \leq 1/2$, but we have not seriously searched for a counterexample. 

\medskip
We proceed to give following interpretation of the latter two corollaries. Denoting:
\[
\delta^0_\gamma(K) = \gamma(K) ~,~ \delta^1_\gamma(K) = \gamma_{\partial K}(\partial K) ~,~ \delta^2_\gamma(K) = \int_{\partial K} H_\gamma d\gamma_{\partial K} ~,
\]
we note that $\delta^i_\gamma(K)$ is precisely the $i$-th variation of the function $t \mapsto \gamma(K_t)$, where $K_t := \set{ x \in \Real^n \; ; \; d(x,K) \leq t}$ and $d$ denotes Euclidean distance. 
Consequently, Corollary \ref{cor:intro-iso-curvature} may be rewritten as:
\begin{cor}[Gaussian iso-second-variation inequality]
If $E$ is a half-plane with $\gamma(E) = \gamma(K) \geq 1/2$, then the following iso-second-variation inequality holds:
\[
\delta^2_\gamma(K) \leq \delta^2_\gamma(E) \; ( \; \leq 0 ) . 
\]
\end{cor}
It is interesting to note that we are not aware of an analogous statement on any other metric-measure space, and in particular, for the Lebesgue measure in Euclidean space, as all known isoperimetric inequalities only pertain (by definition) to the first-variation (and with reversed direction of the inequality). 
Furthermore, in contrast to the isoperimetric inequality, it is easy to see that the second-variation inequality above is false without the assumption that $K$ is convex, as witnessed for instance by taking the complement of any non-degenerate slab $\set{ x \in \Real^n \; ; \; a \leq x_1 \leq b }$ of measure $1/2$. 
As for Corollary \ref{cor:intro-mean-curvature}, we see that it may be rewritten as:

\begin{cor}[Minkowski's second inequality for Euclidean Gaussian extensions]
\begin{equation} \label{eq:Minkowski-G}
\delta^2_\gamma(K) \leq  (\log I_\gamma)'(\delta^0_\gamma(K)) (\delta^1_\gamma(K))^2  .
\end{equation}
\end{cor}

This already hints at the proof of Theorem \ref{thm:main}. To describe the proof, and put the latter interpretation in the appropriate context, let us recall some classical facts from the Brunn-Minkowski theory (for the Lebesgue measure). 

\subsection{Brunn--Minkowski inequality}

The Brunn--Minkowski inequality \cite{Schneider-Book,GardnerSurveyInBAMS} asserts that:
\begin{equation} \label{eq:BM-intro}
Vol((1-t) K + t L)^{1/n} \geq (1-t) Vol(K)^{1/n} + t Vol(L)^{1/n} ~,~ \forall t \in [0,1] ~,
\end{equation}
for all convex $K,L \subset \Real^n$; it was extended to arbitrary Borel sets by Lyusternik. Here $Vol$ denotes Lebesgue measure and 
$A + B := \set{a + b \; ; \; a \in A , b \in B}$ denotes Minkowski addition. We refer to the excellent survey by R. Gardner \cite{GardnerSurveyInBAMS} for additional details and references. 

For convex sets, (\ref{eq:BM-intro}) is equivalent to the concavity of the function $t \mapsto Vol(K + t L)^{1/n}$. 
By Minkowski's theorem, extending Steiner's observation for the case that $L$ is the Euclidean ball, $Vol(K+t L)$ is an $n$-degree polynomial $\sum_{i=0}^n {n \choose i} W_{n-i}(K,L) t^i$, whose coefficients
\begin{equation} \label{eq:W-def}
 W_{n-i}(K,L) := \frac{(n-i)!}{n!} \left .   \brac{\frac{d}{dt}}^{i} \right |_{t=0} Vol(K + t L) ~,
\end{equation}
are called mixed-volumes. The above concavity thus amounts to the following ``Minkowski's second inequality", which is a particular case of the Alexandrov--Fenchel inequalities:
\begin{equation} \label{eq:Mink-II}
W_{n-1}(K,L)^2 \geq W_{n-2}(K,L) W_n(K,L) ~.
\end{equation}
Specializing to the case that $L$ is the Euclidean unit-ball $D$, noting that $K_t = K + t D$, and denoting by $\delta^i(K)$ the $i$-th variation of $t \mapsto Vol(K_t)$, we have as before:
\[
\delta^0(K) = Vol(K) ~,~ \delta^1(K) = Vol_{\partial K}(\partial K) ~,~ \delta^2(K) = \int_{\partial K} H dVol_{\partial K} .
\]
The corresponding distinguished mixed-volumes $W_{n-i}(K) = W_{n-i}(K,D)$, which are called intrinsic-volumes or quermassintegrals,
are related to $\delta^i(K)$ via (\ref{eq:W-def}). Consequently, when $L = D$, Minkowski's second inequality amounts to the inequality:
\[
\delta^2(K) \leq\frac{n-1}{n} \frac{1}{\delta^0(K)} (\delta^1(K))^2 .
\]
The analogy with (\ref{eq:Minkowski-G}) becomes apparent, in view of the fact that $(\log I)'(v) = \frac{n-1}{n} \frac{1}{v}$, where $I(v) = c_n v^{\frac{n-1}{n}}$ is the standard isoperimetric profile of Euclidean space $(\Real^n,\scalar{\cdot,\cdot})$ endowed with the Lebesgue measure. 

\medskip

An important difference to note with respect to the classical theory, is that in the Gaussian theory, $\delta^2_\gamma(K)$ may actually be negative, in contrast to the non-negativity of all mixed-volumes, and in particular of $\delta^2(K)$. One reason for this is that the Gaussian measure is finite whereas the Lebesgue measure is not, so that $I$ is monotone increasing whereas $I_\gamma$ is not. This feature seems to also be responsible for the peculiar iso-second-variation corollary.

\subsection{Ehrhard inequality}

A remarkable extension of the Brunn-Minkowski inequality to the Gaussian setting was obtained by Ehrhard \cite{EhrhardPhiConcavity}, who showed that:
\[
\Phi^{-1}(\gamma((1-t) K + t L)) \geq (1-t) \Phi^{-1}(\gamma(K)) + t \Phi^{-1}(\gamma(L)) \;\;\; \forall t \in [0,1] ,
\]
for all convex sets $K,L \subset \Real^n$, with equality when $K$ and $L$ are parallel half-planes (pointing in the same direction). This was later extended by Lata{\l}a \cite{Latala-EhrhardForOneConvexSet} to the case that only one of the sets is assumed convex, and finally by Borell \cite{Borell-EhrhardForBorelSets,Borell-EhrhardInExtendedRange} to arbitrary Borel sets. As before, for $K,L$ convex sets, Ehrhard's inequality is equivalent to the concavity of the function $t \mapsto F_\gamma(t) := \Phi^{-1}(\gamma((1-t)K + t L))$. 

\medskip

To prove Theorem \ref{thm:main}, we repeat an idea of A.~Colesanti. In \cite{ColesantiPoincareInequality} (see also \cite{ColesantiEugenia-PoincareFromAF}), Colesanti showed that the Brunn-Minkowski concavity of $t \mapsto F(t) := Vol((1-t)K + t L)^{1/n}$ is equivalent to a certain Poincar\'e-type inequality on $\partial K$, by parametrizing $K,L$ via their support functions and calculating the second variation of $F(t)$. Repeating the calculation for $F_\gamma(t)$, Theorem \ref{thm:main} turns out to be an equivalent infinitesimal reformulation of Ehrhard's inequality for convex sets. 

\subsection{Comparison with Previous Results} 

Going in the other direction, we have recently shown in our previous work \cite{KolesnikovEMilmanReillyPart2} how to directly derive a Poincar\'e-type inequality on the boundary of a locally-convex subset of a weighted Riemannian manifold, which may then be used to infer a Brunn-Minkowski inequality in the weighted Riemannian setting via an appropriate geometric flow. In particular, in the Euclidean setting, our results apply to Borell's class of $1/N$-concave measures \cite{BorellConvexMeasures} ($\frac{1}{N} \in [-\infty,\frac{1}{n}]$), defined as those measures $\mu$ on $\Real^n$ satisfying the following generalized Brunn-Minkowski inequality:
\[
\mu((1-t) A+ t B) \geq \brac{(1-t) \mu(A)^{1/N} + t \mu(B)^{1/N}}^N ,
\]
for all $t \in [0,1]$ and Borel sets $A,B \subset \Real^n$ with $\mu(A),\mu(B) > 0$. It was shown by Brascamp--Lieb \cite{BrascampLiebPLandLambda1} and Borell \cite{BorellConvexMeasures} that the absolutely continuous members of this class are precisely characterized by having density $\Psi$ so that $(N-n) \Psi^{1/(N-n)}$ is concave on its convex support $\Omega$ (interpreted as $\log \Psi$ being concave when $N=\infty$), amounting to the Bakry--\'Emery $CD(0,N)$ condition \cite{BakryEmery,BakryStFlour,KolesnikovEMilmanReillyPart1}. Our results from \cite{KolesnikovEMilmanReillyPart2} then imply that for any (say) compact convex $K$ in the interior of $\Omega$ with $C^2$ boundary, and any $f \in C^1(\partial K)$, one has:
\begin{equation} \label{eq:prev0}
\int_{\partial K} H_\mu f^2 d\mu_{\partial K} - \frac{N-1}{N}\frac{1}{\mu(K)} \brac{\int _{\partial K} f d\mu_{\partial K}}^2 \leq  \int_{\partial K} \scalar{\II_{\partial K}^{-1} \;\nabla_{\partial K} f,\nabla_{\partial K} f} d\mu_{\partial K} ~,
\end{equation}
with $\mu_{\partial K} = \Psi \vol_{\partial K}$ denoting the boundary measure and $H_\mu = H + \scalar{\log \Psi , \nu}$ the $\mu$-weighted mean-curvature. Note that the Gaussian measure $\gamma$ satisfies $CD(1,\infty)$ and in particular $CD(0,\infty)$, as $\log \Psi_\gamma$ is concave on $\Real^n$. Consequently, applying (\ref{eq:prev0}) with $N=\infty$, we have:
\begin{equation} \label{eq:prev}
\int_{\partial K} H_\mu f^2 d\gamma_{\partial K} - \frac{1}{\gamma(K)} \brac{\int _{\partial K} f d\gamma_{\partial K}}^2 \leq  \int_{\partial K} \scalar{\II_{\partial K}^{-1} \;\nabla_{\partial K} f,\nabla_{\partial K} f} d\gamma_{\partial K} ~.
\end{equation}
It is easy to verify that $(\log I_\gamma)'(v) < \frac{1}{v}$ for all $v \in (0,1)$, and hence Theorem \ref{thm:main} constitutes an improvement over (\ref{eq:prev}). 

\medskip

A very important point is that the latter improvement is strict only for test functions $f$ with non-zero mean, $\int_{\partial K} f d\gamma_{\partial K} \neq 0$. Put differently, the entire significance of Theorem \ref{thm:main} lies in the coefficient infront of the $\brac{\int _{\partial K} f d\gamma_{\partial K}}^2$ term, since by (\ref{eq:prev0}), for zero-mean test functions, the inequality asserted in Theorem \ref{thm:main} holds not only for the Gaussian measure, but in fact for Borell's entire class of concave (or $CD(0,0)$) measures (using our convention from \cite{KolesnikovEMilmanReillyPart1,KolesnikovEMilmanReillyPart2} that $\frac{N-1}{N} = -\infty$ when $N=0$ and that  $\infty \cdot 0= 0$). 

\medskip

Unfortunately, our method from \cite{KolesnikovEMilmanReillyPart2}, involving $L^2$-duality and the Reilly formula from Riemannian geometry, cannot be used in the Gaussian setting without some additional ingredients, like information on an associated Neumann-to-Dirichlet operator, see Section \ref{sec:3}. In particular, we observe in Section \ref{sec:4} that Theorem \ref{thm:main} (or equivalently, Ehrhard's inequality for convex sets) and even Corollary \ref{cor:intro-mean-curvature}, are simply false for a general $CD(1,\infty)$ probability measure in Euclidean space, having density $\Psi = \exp(-V)$ with $\nabla^2 V \geq Id$.

\section{Proof of Theorem \ref{thm:main}} \label{sec:2}

The general formulation of Theorem \ref{thm:main} is reduced to the case that $K$ is compact with strictly-convex $C^3$ smooth boundary ($\II_{\partial K} > 0$) by a standard (Euclidean) approximation argument - this class of convex sets is denoted by $\C^3_+$ (and analogously we define the class $\C^2_+$). As explained in the Introduction, the proof of Theorem \ref{thm:main} boils down to a direct calculation of the second variation of the function:
\[
 t \mapsto \Phi^{-1}(\gamma((1-t)K + t L)) 
\]
for an appropriately chosen $L$. Ehrhard's inequality ensures that this function is concave when $K,L$ are convex. 

The second variation will be conveniently expressed using support functions. Recall that the support function of a convex body (convex compact set with non-empty interior) $C$ is defined as the following function on the Euclidean unit sphere $S^{n-1}$:
\[
h_C(\theta) := \sup \set{ \scalar{\theta,x} ; x \in C } ~,~ \theta \in S^{n-1}.
\]
It is easy to see that the correspondence $C \mapsto h_C$ between convex bodies and functions on $S^{n-1}$ is injective and positively linear: $h_{a C_1 + b C_2} = a h_{C_1} + b h_{C_2}$ for all $a,b \geq 0$. 
As $K \in \C^3_+$ we know that $h_K$ is $C^3$ smooth \cite[p. 106]{Schneider-Book}. 

Now let $f  \in C^2(\partial K)$, and consider the function $h_{\Delta} := f \circ \nu_{\partial K}^{-1} : S^{n-1} \rightarrow \Real$. Since $K \in \C^3_+$ this function is well-defined and $C^2$ smooth. Moreover, it is not hard to show (e.g. \cite[pp. 38, 111]{Schneider-Book},\cite{ColesantiPoincareInequality}) that for $\eps > 0$ small enough, $h_{K} + t h_{\Delta}$ is the support function of a convex body $K_t \in \C^2_+$ for all $t \in [0,\eps]$. It follows by linearity of the support functions that $K_{\eps t} = (1-t) K + t K_\eps$ for all $t \in [0,1]$, and so Ehrhard's inequality implies that:
\[
 t \mapsto F_\gamma(t) := \Phi^{-1}(\gamma(K_t)) 
\]
is concave on $[0,\eps]$. 

The first and second variations of $t \mapsto \mu(K_t)$ were calculated by Colesanti in \cite{ColesantiPoincareInequality} for the case that $\mu$ is the Lebesgue measure, and for general measures $\mu$ with positive density $\Psi$ by the authors in \cite{KolesnikovEMilmanReillyPart2} (in fact in a general weighted Riemannian setting, with an appropriate interpretation of $K_t$ avoiding support functions):
\begin{align*}
\delta^0 & := (d/dt)^0 |_{t=0} \; \mu(K_t) = \mu(K) ~, \\ 
\delta^1 & := (d/dt)^1 |_{t=0} \; \mu(K_t)  =  \int_{\partial K} f d\mu_{\partial K} ~, \\
\delta^2 & := (d/dt)^2 |_{t=0} \; \mu(K_t)  =  \int_{\partial K} H_\mu  f^2 d\mu_{\partial K} - \int_{\partial K} \scalar{\II_{\partial K}^{-1} \nabla_{\partial K} f, \nabla_{\partial K} f } d\mu_{\partial K} .
\end{align*}
Applying the above formulae for $\mu = \gamma$, calculating:
\[
0 \geq F_\gamma''(0) = (\Phi^{-1})''(\delta^0) (\delta^1)^2 + (\Phi^{-1})'(\delta^0) \delta^2  ,
\]
dividing by $(\Phi^{-1})'(\delta^0) > 0$ and using that:
\[
\frac{(\Phi^{-1})''(v)}{(\Phi^{-1})'(v)} = - (\log I_\gamma)'(v) ,
\]
Theorem \ref{thm:main} readily follows for $f \in C^2(\partial K)$. The general case for $f \in C^1(\partial K)$ is obtained by a standard approximation argument.

\medskip

Going in the other direction, it should already be clear that Theorem \ref{thm:main} implies back Ehrhard's inequality. Indeed, given $K,L \in C^2_+$, consider $K_t = (1-t) K + t L$ for $t \in [0,1]$, and note that $h_{K_t} = (1-t) h_K + t h_L$. Fixing $t_0 \in (0,1)$, it follows that $h_{K_{t_0 + \eps}} = h_{K_{t_0}} + \eps (h_L - h_K)$. Inspecting the proof above, we see that the statement of Theorem \ref{thm:main} for $K_{t_0}$ and $f = (h_L - h_K) \circ \nu \in C^1(\partial K)$, is precisely equivalent to the concavity of the function $\eps \mapsto \Phi^{-1}(\gamma(K_{t_0 + \eps}))$ at $\eps = 0$. Since the point $t_0 \in (0,1)$ was arbitrary, we see that Theorem \ref{thm:main} implies the concavity of $[0,1] \ni t \mapsto \Phi^{-1}(\gamma((1-t) K + t L))$ for $K,L \in C^2_+$. The case of general convex $K,L$ follows by approximation.

\section{Neumann-to-Dirichlet Operator} \label{sec:3}

In this section, we mention how a certain property of a Neumann-to-Dirichlet operator can be used to directly obtain an Ehrhard-type inequality for general measures $\mu = \exp(-V(x)) dx$ on $\Real^n$ (say with $C^2$ positive density). Define the associated weighted Laplacian $L = L_\mu$ as:
\[
L  = L_\mu := \exp(V) \nabla \cdot ( \exp(-V) \nabla) = \Delta - \scalar{\nabla V,\nabla} ~.
\]
Given a compact set $\Omega \subset \Real^n$ with $C^1$ smooth boundary, note that the usual integration by parts formula is satisfied for $f,g \in C^2(\Omega)$: \[
\int_\Omega L(f) g d\mu = \int_{\partial M} f_\nu g d\mu_{\partial M} - \int_\Omega \scalar{\nabla f,\nabla g} d\mu = \int_{\partial \Omega} (f_\nu g - g_\nu f) d\mu_{\partial \Omega} +  \int_\Omega L(g) f d\mu ~,
\]
where  $u_\nu = \nu \cdot u$. 

Given a compact convex body $K \in C^2_+$ and $f \in C^{1,\alpha}(\partial K)$, let us now solve the following Neumann Laplace equation:
\begin{equation} \label{eq:Laplace}
L u \equiv \frac{1}{\mu(K)} \int_{\partial K} f d\mu_{\partial K}  \text{ on $K$ } ~,~  u_\nu = f  \text{ on $\partial K$} .
\end{equation}
Since the compatibility condition $\int_K L u d\mu = \int_{\partial K} f d\mu_{\partial K}$ is satisfied, it is known (e.g. \cite{KolesnikovEMilmanReillyPart2}) that a solution $u \in C^{2,\alpha}(K)$ exists (and is unique up to an additive constant). The operator mapping $f \mapsto u$ is called the Neumann-to-Dirichlet operator. 

\begin{thm} \label{thm:D2N}
Assume that there exists a function $F : \Real_+ \rightarrow \Real$ so that for all $K$, $f$ and $u$ as above:
\begin{equation} \label{eq:D2N-assumption}
F(\mu(K)) (\int_{\partial K} f d\mu_{\partial K})^2 \leq \int_{\partial K} \brac{ \scalar{\nabla_{\partial K} f , \nabla_{\partial K} u} + u_{\nu,\nu} f} d\mu_{\partial K} .
\end{equation}
Denote $G(v) := \frac{1}{v} - F(v)$ and $\Phi_\mu^{-1}(v) := \int_{1/2}^v \exp(- \int_{1/2}^t G(s) ds) dt$. Then for all $f \in C^1(\partial K)$:
\begin{equation} \label{eq:D2N-assertion1}
\int_{\partial K} H_\mu f^2 d\mu_{\partial K} - G(\mu(K)) \brac{\int_{\partial K} f d\mu_{\partial K}}^2  \leq \int_{\partial K} \scalar{\II_{\partial K}^{-1} \nabla_{\partial K} f , \nabla_{\partial K} f} d\mu_{\partial K} ,
\end{equation}
and for all convex $K,L \subset \Real^n$ and $t \in [0,1]$:
\begin{equation}  \label{eq:D2N-assertion2}
\Phi_\mu^{-1}((1-t) K + t L) \geq (1-t) \Phi_\mu^{-1}(K) + t \Phi_\mu^{-1}(L) .
\end{equation}
\end{thm}

For the proof, we require the following lemma. We denote by $\norm{\nabla^2 u}$ the Hilbert-Schmidt norm of the Hessian $\nabla^2 u$. 

\begin{lem} \label{lem:D2N}
With the above notation:
\[
\int_K \brac{\scalar{ \nabla^2 V \; \nabla u , \nabla u} + \norm{\nabla^2 u}^2} d\mu = \int_{\partial K} \brac{ \scalar{\nabla_{\partial K} f , \nabla_{\partial K} u} + u_{\nu,\nu} f} d\mu_{\partial K} .
\]
\end{lem}
\begin{rem}
The integrand on the left-hand-side above is the celebrated Bakry--\'Emery iterated carr\'e-du-champ $\Gamma_2(u)$, associated to $(K,\scalar{\cdot,\cdot},\mu)$ \cite{BakryEmery}.
\end{rem}
\begin{proof}
Denoting $\nabla u = (u_1,\ldots,u_n)$, we calculate:
\[
\int_K \norm{\nabla^2 u}^2 d\mu = \int_K \sum_{i=1}^n \abs{\nabla u_i}^2 d\mu = \int_{\partial K} \sum_{i=1}^n u_{i,\nu} u_i d\mu_{\partial K} - \int_K \sum_{i=1}^n  L(u_i) u_i d\mu . 
\]
To handle the $L(u_i)$ terms, we take the $i$-th partial derivative in the Laplace equation (\ref{eq:Laplace}), yielding:
\[
0 = (L u)_i  = L(u_i) - \scalar{\nabla u, \nabla V_i} .
\]
Consequently, we have:
\[
\sum_{i=1}^n L(u_i) u_i = \scalar{\nabla^2 V \; \nabla u , \nabla u} ,
\]
and therefore:
\[
\int_K \brac{\scalar{ \nabla^2 V \; \nabla u , \nabla u} + \norm{\nabla^2 u}^2} d\mu = \int_{\partial K} \sum_{i=1}^n u_{i,\nu} u_i d\mu_{\partial K} .
\]
Recalling that $f = u_\nu$, the assertion follows. 
\end{proof}

\begin{proof}[Proof of Theorem \ref{thm:D2N}]
As in \cite{KolesnikovEMilmanReillyPart2}, our starting point is the generalized Reilly formula \cite{KolesnikovEMilmanReillyPart1}, which is an integrated form of Bochner's formula in the presence of a boundary. In the Euclidean setting, it states that for any $u \in C^2(K)$ (see \cite{KolesnikovEMilmanReillyPart1} for less restrictions on $u$):
\begin{multline}
\label{Reilly}
\int_K (L u)^2 d\mu = \int_K \norm{\nabla^2 u}^2 d\mu + \int_K \scalar{ \nabla^2 V\; \nabla u, \nabla u} d\mu + \\
\int_{\partial K} H_\mu (u_\nu)^2 d\mu_{\partial K} + \int_{\partial K} \scalar{\II_{\partial K}  \;\nabla_{\partial K} u,\nabla_{\partial K} u} d\mu_{\partial K} - 2 \int_{\partial K} \scalar{\nabla_{\partial K} u_\nu, \nabla_{\partial K} u} d\mu_{\partial K} ~.
\end{multline}
As we assume that $\II_{\partial K} > 0$, we may apply the Cauchy--Schwarz inequality to the last-term above:
\begin{equation} \label{eq:CS}
2 \scalar{\nabla_{\partial K} u_\nu, \nabla_{\partial K} u} \leq \scalar{\II_{\partial K} \; \nabla_{\partial K} u , \nabla_{\partial K} u} + \scalar{\II_{\partial K}^{-1} \nabla_{\partial K} u_\nu , \nabla_{\partial K} u_\nu} ~,
\end{equation}
yielding:
\begin{align*}
\int_K (L u)^2 d\mu &\geq \int_K \norm{\nabla^2 u}^2 d\mu + \int_K \scalar{ \nabla^2 V\; \nabla u, \nabla u} d\mu \\
& + \int_{\partial K} H_\mu (u_\nu)^2 d\mu_{\partial K} - \int_{\partial K} \scalar{\II_{\partial K}^{-1} \;\nabla_{\partial K} u_\nu,\nabla_{\partial K} u_\nu} d\mu_{\partial K} ~.
\end{align*}

Given $f \in C^{1,\alpha}(\partial K)$, we now apply the above inequality to the solution $u$ of the Neumann Laplace equation (\ref{eq:Laplace}). Together with Lemma \ref{lem:D2N}, this yields:
\begin{align*}
 \int_{\partial K} H_\mu (u_\nu)^2 d\mu_{\partial K} - \frac{1}{\mu(K)} \brac{\int_K f d\mu}^2  & + \int_{\partial K} \brac{ \scalar{\nabla_{\partial K} f , \nabla_{\partial K} u} + u_{\nu,\nu} f} d\mu_{\partial K} \\
&  \leq \int_{\partial K} \scalar{\II_{\partial K}^{-1} \;\nabla_{\partial K} f,\nabla_{\partial K} f} d\mu_{\partial K} .
\end{align*}
Invoking our assumption (\ref{eq:D2N-assumption}), the asserted inequality (\ref{eq:D2N-assertion1}) follows for $f \in C^{1,\alpha}(\partial K)$. The case of a general $f \in C^1(\partial K)$ follows by a standard approximation argument.

Lastly, (\ref{eq:D2N-assertion2}) is an equivalent version of (\ref{eq:D2N-assertion1}). Indeed, the proof provided in Section \ref{sec:2} demonstrates how to pass from (\ref{eq:D2N-assertion2}) to (\ref{eq:D2N-assertion1}), with:
\[
G(v) =   \frac{(\Phi_\mu^{-1})''(v)}{(\Phi_\mu^{-1})'(v)} = (\log ((\Phi_\mu^{-1})'))'(v) .
\]
To see the other direction, repeat the argument described in the previous section. After establishing (\ref{eq:D2N-assertion2}) for $K,L \in C^2_+$, the general case follows by a standard approximation argument. 
\end{proof}

Unfortunately, we cannot claim that condition (\ref{eq:D2N-assumption}) is equivalent to the Ehrhard-type inequality (\ref{eq:D2N-assertion2}), since the proof of Theorem \ref{thm:D2N} involved an application of the Cauchy-Schwartz inequality (\ref{eq:CS}). Consequently, we pose this as a question:

\begin{question}[Gaussian Neumann-to-Dirichlet Operator on Convex Domains]  \label{ques}
Does (\ref{eq:D2N-assumption}) hold for $\mu=\gamma$ the Gaussian measure with $F(v) = \frac{1}{v} - (\log I_\gamma)'(v)$? 
\end{question}

Note that the analogous question for $\frac{1}{N}$-concave measures $\mu$, $\frac{1}{N} \in (-\infty,\frac{1}{n}]$, has a positive answer: (\ref{eq:D2N-assumption}) holds for any $K \in C^2_+$ in the support of $\mu$ with $F(v) = \frac{1}{N} \frac{1}{v}$. Indeed, if $\mu = \Psi(x) dx = \exp(-V(x)) dx$ satisfies on its support:
\[
- (N-n) \frac{\nabla^2 \Psi^{\frac{1}{N-n}}}{\Psi^{\frac{1}{N-n}}} = \nabla^2 V - \frac{1}{N-n} \nabla V \otimes \nabla V \geq 0 ,
\]
then by several applications of the Cauchy--Schwartz inequality (see \cite{KolesnikovEMilmanReillyPart1}):
\begin{align*}
& \int_{\partial K} \brac{ \scalar{\nabla_{\partial K} f , \nabla_{\partial K} u} + u_{\nu,\nu} f} d\mu_{\partial K} = \int_K \brac{\scalar{ \nabla^2 V \; \nabla u , \nabla u} + \norm{\nabla^2 u}^2} d\mu \\
& \geq \int_K  \brac{\frac{1}{N-n} \scalar{\nabla u, \nabla V}^2 + \frac{1}{n} (\Delta u)^2} d\mu \geq \int_K \frac{1}{N} (\Delta u - \scalar{\nabla u,\nabla V})^2 d\mu \\
& = \frac{1}{N} \int_K (Lu)^2 d\mu = \frac{1}{N} \frac{1}{\mu(K)} (\int_{\partial K} f d\mu_{\partial K})^2 .
\end{align*}

\section{Concluding Remarks} \label{sec:4}

\subsection{Refined Version}

Peculiarly, as in \cite{KolesnikovEMilmanReillyPart2}, it is possible to strengthen Theorem \ref{thm:main} by applying it to $f+z$ and optimizing over $z \in \Real$.
This results in the following stronger inequality: \begin{align*}
&\int_{\partial M} H_\gamma f^2 d\mu_{\partial M} - (\log I_\gamma)'(\gamma(K)) (\int_{\partial K} f d\gamma_{\partial K})^2 + \frac{\brac{ \int_{\partial K} f \beta d\gamma_{\partial K} }^2}{\int_{\partial K} \beta d\gamma_{\partial K}} \\
& \leq  \int_{\partial K} \scalar{\II_{\partial K}^{-1} \;\nabla_{\partial K} f,\nabla_{\partial K} f} d\gamma_{\partial K} ~,
\end{align*}
where:
\[
\beta(x) := (\log I_\gamma)'(\gamma(K)) \gamma_{\partial K}(\partial K) - H_\gamma(x) ~.
\]
Note that indeed $\int_{\partial K} \beta d\gamma_{\partial K} \geq 0$ by Corollary \ref{cor:intro-mean-curvature}, so the additional third term appearing above is always non-negative.

Recall that our original weaker inequality (\ref{eq:main}) is an equivalent infinitesimal form of Ehrhard's inequality, and so one cannot hope to obtain a strict improvement in the cases when Ehrhard's inequality is sharp (and indeed when $K$ is a half-plane we see that $\beta \equiv 0$). 
On the other hand, it would be interesting to integrate back the stronger inequality above and obtain a refined version of Ehrhard's inequality, which would perhaps be better suited for obtaining delicate stability results (cf. \cite{MosselNeeman-StabilityOfGaussianIsoperimetry,EldanGaussianNoiseStabilityDeficit} and the references therein). We leave this for another occasion. 
\subsection{Mean-Curvature Inequality implies Isoperimetric Inequality}

As explained in Section \ref{sec:2},  Theorem \ref{thm:main} is an equivalent infinitesimal form of Ehrhard's inequality (for convex domains $K,L$), i.e. equivalent to the concavity of $[0,1] \ni t \mapsto \Phi^{-1}(\gamma((1-t) K + t L))$. Similarly, Corollary \ref{cor:intro-mean-curvature}, which is obtained by setting $f\equiv1$ in Theorem \ref{thm:main}, is an equivalent infinitesimal form of the concavity of:
\[
 \Real_+ \ni t \mapsto F(t) := \Phi^{-1}(\gamma(K + t B_2^n)) ,
\]
where $B_2^n$ denotes the Euclidean unit-ball; indeed, Corollary \ref{cor:intro-mean-curvature} expresses precisely that $F''(0) \leq 0$. 

It is worthwhile to note that the latter concavity may be used to recover the Gaussian isoperimetric inequality (albeit only for convex sets). The following is a variant on an argument due to Ledoux (private communication), who showed how Ehrhard's inequality with $L$ being a multiple of $B_2^n$, may be used to recover the Gaussian isoperimetric inequality (for general Borel sets). Indeed, the concavity of $F$ implies that:
\[
F'(0) \geq \lim_{t \mapsto \infty} \frac{F(t) - F(0)}{t} = \lim_{t \mapsto \infty} \frac{F(t)}{t} \geq \lim_{t \rightarrow \infty} \frac{\Phi^{-1}(\gamma(t B_2^n))}{t}  . 
\]
A straightforward calculation (e.g. \cite{BogachevGaussianBook}) verifies that the right-hand-side is equal to $1$, and hence:
\[
1 \leq F'(0) = (\Phi^{-1})'(\gamma(K)) \gamma_{\partial K}(\partial K) , 
\]
or equivalently:
\[
\gamma_{\partial K}(\partial K) \geq I_\gamma(\gamma(K)) ,
\]
as asserted.

\subsection{Ehrhard's inequality is false for $CD(1,\infty)$ measures}

It is well known (e.g. \cite{Ledoux-Book}) that various isoperimetric, functional and concentration inequalities which are valid for the Gaussian measure are also valid for any measure $\mu = \exp(-V) dx$ on $\Real^n$ with $\nabla^2 V \geq Id$, the so-called class of $CD(1,\infty)$ measures in Euclidean space.

However, we remark that it is not possible to extend Ehrhard's inequality (and hence Theorem \ref{thm:main}) to this more general class, providing in particular a negative answer to Question \ref{ques} for several natural members of this class. Indeed, this is witnessed already by considering the probability measure $\mu$ obtained by conditioning the one-dimensional Gaussian measure onto a half-line $(-\infty,b]$ (which may clearly be approximated in total-variation by probability measures $\exp(-V) dx$ with $V'' \geq 1$). It is not true that:
\[
\Phi^{-1}(\mu((1-t) K + t L)) \geq (1-t) \Phi^{-1}(\mu(K)) + t \Phi^{-1}(\mu(L)) ,
\]
even for half-lines $K,L$. If that were the case, it would mean that the function $(-\infty,b] \ni t \mapsto \Phi^{-1}(\Phi(t)/\Phi(b))$ is concave, but it is easy to see that this is not the case as $t \rightarrow b$. The same argument shows that $\Real_+ \ni t \mapsto \Phi^{-1}(\mu(K + t [-1,1]))$ is not concave even for a half-line $K$, and so we see that even Corollary \ref{cor:intro-mean-curvature} cannot be extended to the $CD(1,\infty)$ setting.

\subsection{Dual Inequality for Mean-Convex Domains}

Lastly, for completeness, we specialize a dual Poincar\'e-type inequality obtained in \cite{KolesnikovEMilmanReillyPart2}, for the case of the Gaussian measure:

\begin{thm}[Dual Inequality for Mean-Convex Domains] 
Let $K \subset \Real^n$ denote a compact set with $C^2$ smooth boundary which is strictly Gaussian mean-convex, i.e. $H_{\gamma} > 0$ on $\partial K$. Then for any $f \in C^{2}(\partial K)$ and $C \in \Real$:
\[
\int_{\partial K} \scalar{\II_{\partial K}  \;\nabla_{\partial K} f,\nabla_{\partial K} f} d\gamma_{\partial K} \leq \int_{\partial K} \frac{1}{H_\gamma} \Bigl(L_{\partial K} f + \frac{ (f-C) }{2} \Bigr)^2 d\gamma_{\partial K} ~.
\]
Here $L_{\partial K}  = \Delta_{\partial K} - \scalar{x,\nabla_{\partial K}}$ denotes the induced Ornstein--Uhlenbeck generator on $\partial K$. 
\end{thm}

\setlinespacing{0.5}
\setlength{\bibspacing}{2pt}

\vspace{-20pt}

\bibliographystyle{plain}
\bibliography{../../../ConvexBib}

\end{document}